\documentclass[a4paper,10pt]{amsart}
\usepackage[utf8]{inputenc}
\usepackage[T2A]{fontenc}
\usepackage[english]{babel}

\usepackage{amssymb}
\usepackage{amsthm}

\newtheorem{lemma}{Lemma}
\newtheorem{theorem}{Theorem}

\def\hm#1{#1\nobreak\discretionary{}{\hbox{\ensuremath{#1}}}{}}

\def\taue{\tau^{(e)}}

\def\eps{\varepsilon}

\def\res{\mathop{\mathrm{res}}}

\def\suma{\mathop{\sum\nolimits^*}}

\def\le{\leqslant}
\def\ge{\geqslant}

\def\half{{\textstyle{1\over2}}}
\def\threehalf{{\textstyle{3\over2}}}

\def\taa{\tau_{1,2}}

\def\suma{\mathop{\sum\nolimits'}\limits}

\usepackage[unicode]{hyperref}
\usepackage{breakurl}
\usepackage{graphics}

\begin{document}

\title{Average number of squares dividing $mn$}
\author{Andrew V. Lelechenko}
\address{I.~I.~Mechnikov Odessa National University}
\email{1@dxdy.ru}


\begin{abstract}
We study the asymptotic behaviour of $\sum_{m,n\le x} \taa(mn)$,
where $\taa(n) \hm= \sum_{a b^2 = n} 1$, using multidimensional Perron formula and complex integration method. An asymptotic formula with an error term $O( x^{10/7})$ is obtained.
\end{abstract}

\maketitle

\section{Introduction}\label{s:introduction}

Let $f$ be a multiplicative arithmetic function of one variable. The asymp\-to\-tic behaviour of $\sum_{n\le x} f(n)$ is a classic problem of analytic number theory, deeply studied for various specific functions and classes. Let us consider the problem of estimating of $\sum_{m,n\le x} f(mn)$.

The divisor function $\tau$ is a simple, but non-trivial case. Applying Busche---Ra\-ma\-nu\-jan identity
\begin{equation}\label{eq:br-identity}
\tau(mn) = \sum_{d \mid \gcd(m,n)} \tau(m/d) \tau(n/d) \mu(d)
\end{equation}
we split variables and obtain
\begin{equation*}
\sum_{m,n \le x} \tau(mn) = \sum_{\substack{j,k,l \\ j,k \le x/l}} \tau(j) \tau(k) \mu(l) = \sum_{l\le x} \mu(l) \biggl( \sum_{j\le x/l} \tau(j) \biggr)^2.
\end{equation*}
Using Huxley's estimate~\cite{huxley2005} $\sum_{j\le y} \tau(j) = y \log y + (2\gamma-1) y + O(y^{\theta+\eps})$, where~$\theta\hm=131/416$, we regroup terms and get
\begin{multline}\label{eq:tau-mn}
\sum_{m,n \le x} \tau(mn) = x^2 \Biggl(
\biggl( \sum_{l=1}^\infty {\mu(l) \over l^2} \biggr)
\biggl( \log^2 x + 2(2\gamma-1)\log x + (2\gamma-1)^2 \biggr)
-{} \\
-
\biggl( \sum_{l=1}^\infty {\mu(l) \log l \over l^2} \biggr)
\biggl( 2 \log x + 2(2\gamma-1) \biggr)
+
\sum_{l=1}^\infty {\mu(l) \log^2 l \over l^2}
\Biggr) + O(x^{1+\theta+\eps}).
\end{multline}

It is natural to ask whether the main term can be derived analytically, by complex integration method. We will not go into details, but note that
$$
\sum_{a,b=0}^\infty \tau(p^{a+b}) x^a y^b
=
\sum_{a,b=0}^\infty (a+b+1) x^a y^b
= {1-xy \over (1-x)^2 (1-y)^2},
\quad
|x|, |y| < 1.
$$
The series $\sum_{m,n=1}^\infty \tau(mn) m^{-z} n^{-w}$ converges absolutely for $\Re z, \Re w > 1$, so by multiplicativity in this region we have
\begin{equation}\label{eq:tau-series}
\sum_{m,n=1}^\infty {\tau(mn) \over m^z n^w}
=
\prod_p \sum_{a,b=0}^\infty { \tau(p^{a+b}) \over p^{az+bw}}
=
\prod_p {{1-p^{-z-w} \over (1-p^{-z})^2 (1-p^{-w})^2} }
=
{\zeta^2(z) \zeta^2(w) \over \zeta(z+w)}.
\end{equation}
Achieved representation allows to compute the coefficient of multiple Laurent series for $x^{z+w} \* z^{-1} \* w^{-1} \* \sum_{m,n=1}^\infty {\tau(mn)  m^{-z} n^{-w}}$ at $1/(z-1)(w-1)$, which ap\-pears coinciding with the main term of~\eqref{eq:tau-mn}.

\medskip

Out paper is devoted to
$$
\sum_{m,n\le x} \taa(mn),
$$
where $\taa(n) = \sum_{a b^2 = n} 1$. This function is not as lucky as $\tau$ and does not posses representation like~\eqref{eq:br-identity}, so there is no easy way to split $m$ and $n$.

The main result is
\begin{theorem}\label{th:main-theorem}
$$
\sum_{m,n\le x} \taa(mn) = C_1 x^2 + C_2 x^{3/2} + O( x^{10/7+\eps}),
$$
where $C_1=2.995\ldots$, $C_2=-5.404\ldots$ are computable constants.
\end{theorem}

This theorem is analogous to the estimate by Graham and Kolesnik~\cite{graham1988}
$$
\sum_{n\le x} \taa(n) = \zeta(2) x + \zeta(1/2) x^{1/2} + O(x^{\beta+\eps}),
\quad
\beta = 1057/4785 \approx 0.2209.
$$

\section{Notations}

Letter $p$ with or without indexes denotes a prime number.
We write~$f\hm\star g$ for the Dirichlet convolution
$$ (f \star g)(n) = \sum_{d|n} f(d) g(n/d). $$

In asymptotic relations we use $\sim$, $\asymp$, Landau symbols $O$ and $o$, Vinogradov symbols $\ll$ and $\gg$ in their usual meanings. All asymptotic relations are given as an argument (usually $x$) tends to the infinity.

Letter $\gamma$ denotes Euler---Mascheroni constant.  Everywhere $\eps>0$ is an arbitrarily small number (not always the same even in one equation).

As usual $\zeta(s)$ is the Riemann zeta-function.
Real and imaginary com\-po\-nents of the complex~$s$ are denoted as $\sigma:=\Re s$ and~$t:=\Im s$, so~$s=\sigma+it$.

For a fixed $\sigma\in[1/2,1]$ define
$$
\mu(\sigma) := \limsup_{t\to\infty} {\log \bigl|\zeta(\sigma+it)\bigr| \over \log t}.
$$

\section{Preliminary estimates}\label{s:preliminary}

We say that a function is symmetric if any permutation of arguments does not change its value.

Let $f$ be an arithmetic function of $r$ variables. The associated Dirichlet series are defined as
$$
F(s_1,\ldots,s_r) = \sum_{n_1,\ldots,n_r=1}^\infty f(n_1,\ldots,n_r) n_1^{-s_1} \cdots n_r^{-s_r}
$$
and a tuple $(\sigma_1,\ldots,\sigma_r)$ is called abscissas of absolute convergence if $F(s_1,\ldots,s_r)$ converges absolutely in the region $\Re s_1 > \sigma_1, \ldots, \Re s_r > \sigma_r$.

\begin{lemma}\label{l:balazard}
Let $f$ be a symmetric arithmetic function of $r$ variables and $(\sigma_a,\ldots,\sigma_a)$ are abscissas of absolute convergence of the associated Dirichlet series $F(s_1,\ldots,s_r)$. Define
\begin{equation}\label{eq:f-heart}
F_r^\heartsuit(\sigma, x, T) :=
\sum_{n_1,\ldots,n_r=1}^\infty {|f(n_1,\ldots,n_r)| (n_1\cdots n_r)^{-\sigma} \over \min_{j=1,\ldots,r} (T|\log(x/n_j)| + 1)}.
\end{equation}
and let
\begin{equation}\label{eq:sum-prime}
\suma_{n_1,\ldots,n_r \le x} f(n_1,\ldots,n_r) := \sum_{n_1,\ldots,n_r \le x} f(n_1,\ldots,n_r) h(x/n_1) \cdots h(x/n_r),
\end{equation}
where $h(y)=0$ for $0<y<1$, $h(1)=1/2$ and $h(y)=1$ otherwise.

For $x\ge2$, $T\ge2$, $\sigma\le\sigma_a$, $\delta>0$, $\kappa=\sigma_a - \sigma + \delta/\log x$, $1=N_1 \hm\le \cdots\le N_r$, $1=M_1\le\cdots\le M_r$ and $N_0:=N_1+\cdots+N_r$ we have
\begin{multline}\label{eq:perron}
\Biggl|
\suma_{n_1,\ldots,n_r \le x} {f(n_1,\ldots,n_r) \over (n_1\cdots n_r)^s}
-{}
\\
-
{1\over(2\pi i)^r}
\!\!\!
\int\limits_{N_1\kappa-iM_1T}^{N_1\kappa+iM_1T}
\!\!\!
\cdots
\!\!\!
\int\limits_{N_r\kappa-iM_rT}^{N_r\kappa+iM_rT}
\!\!\!
F(s+w_1,\ldots,s+w_r) x^{w_1+\cdots+w_r}
{ {d w_1\cdots d w_r \over w_1 \cdots w_r}}
\Biggr|
\ll{}
\\
\ll
x^{N_0(\sigma_a-\sigma)}
F_r^\heartsuit(\sigma_a+\delta/\log x, x, T).
\end{multline}
\end{lemma}
\begin{proof}
This is a result of Balazard, Naimi and Pétermann \cite[Prop.~6]{balazard2008}.
\end{proof}

\begin{lemma}\label{l:log-sum}
Let $f(t)\ge 0$. If
$$ \int_1^T f(t) \, dt \ll g(T), $$
where $g(T) = T^\alpha \log^\beta T$, $\alpha\ge 1$,
then
\begin{equation*}
I(T):= \int_1^T {f(t)\over t} dt \ll
\left\{ \begin{matrix}
\log^{\beta+1} T & \text{if } \alpha=1, \\
T^{\alpha-1} \log^{\beta} T & \text{if } \alpha>1.
\end{matrix} \right.
\end{equation*}
\end{lemma}

\begin{proof}
Let us divide the interval of integration into parts:
\begin{multline*}
I(T)
\le
\sum_{k=0}^{\lfloor \log_2 T \rfloor - 1}
\int_{T/2^{k+1}}^{T/2^k} {f(t)\over t} dt
+g(2)
< \\ <
\sum_{k=0}^{\log_2 T} {1\over T/2^{k+1}}
\int_1^{T/2^k}
f(t) \,dt + g(2)
\ll
\sum_{k=0}^{\lfloor \log_2 T \rfloor - 1}
{g(T/2^{k})\over T/2^{k+1}}.
\end{multline*}
Now the lemma's statement follows from elementary estimates.
\end{proof}

\begin{lemma}\label{l:phragmen}
Let~$\eta>0$ be arbitrarily small. Then for growing $|t|\ge3$
\begin{equation}\label{eq:convexity}
\zeta(s) \ll \begin{cases}
|t|^{1/2 - (1-2\mu(1/2))\sigma}, & \sigma\in[0, 1/2],
\\
|t|^{2\mu(1/2)(1-\sigma)} , & \sigma\in[1/2, 1-\eta], \\
|t|^{2\mu(1/2)(1-\sigma)} \log^{2/3} |t| , & \sigma\in[1-\eta, 1], \\
\log^{2/3} |t|, & \sigma\in[1,1+\eta], \\
1, & \sigma\ge1+\eta.
\end{cases}
\end{equation}
\end{lemma}

\begin{proof}
Estimates follow from Phragmén---Lindelöf principle and estimates of~$\zeta(s)$ at $\sigma=0, 1/2, 1$.
See Titchmarsh~\cite[Ch.~5]{titchmarsh1986} or Ivić~\cite[Ch.~7.5]{ivic2003} for details.
\end{proof}

\begin{lemma}\label{l:zeta-square}
$$
\int_1^T \bigl| \zeta(\sigma+it) \bigr|^2 dt \ll T,
\qquad
1/2 < \sigma < 1.
$$
\end{lemma}
\begin{proof}
See Ivić~\cite[(1.76)]{ivic2003}.
\end{proof}

\section{Reduction to complex integration}

Applying Lemma~\ref{l:balazard} with $r=2$, $f(n_1,n_2)=\taa(n_1 n_2)$, $\sigma=s=0$, $\sigma_a=1$, $N_1\hm=N_2\hm=M_1=M_2=1$, $\delta=1$, $\log T \asymp \log x$ and writing $(m,n,z,w,c)$ instead of~$(n_1,n_2,w_1,w_2,\kappa)$ for convenience we deduce from~\eqref{eq:perron} that
\begin{equation}\label{eq:perron-tau12}
\suma_{m,n\le x} \taa(mn) =
{1\over(2\pi i)^2} \iint\limits_{[c-iT,c+iT]^2} F(z,w) {x^{z+w} \over zw} dz \, dw
+
O\left(
x^{2}
F_2^\heartsuit(c, x, T)
\right),
\end{equation}
where $c=1+1/\log x$ and
\begin{equation}\label{eq:tau12-series}
F(z,w) = \sum_{m,n=1}^\infty {\taa(mn) \over m^z n^w},
\qquad
\Re z, \Re w > 1.
\end{equation}

By~\eqref{eq:f-heart} for non-integer $x$
\begin{multline}\label{eq:sigma1234}
T F_2^\heartsuit(c, x, T)
\ll
\sum_{m,n} {\taa(mn) \over (mn)^c \min\bigl( |\log {x\over n}|, |\log {x\over m}|  \bigr)}
\ll
\sum_{\substack{|\log {x\over n}| \ge 1 \\ |\log {x\over m}| \ge 1}}
	{\taa(mn) \over (mn)^c}
+ \\ +
\sum_{\substack{|\log {x\over n}| \le 1 \\ |\log {x\over m}| \ge 1}}
{\taa(mn) \over (mn)^c |\log {x\over n}|}
+
\sum_{\substack{|\log {x\over n}| \le 1 \\ |\log {x\over m}| \le 1}}
{\taa(mn) \over (mn)^c \min\bigl( |\log {x\over n}|, |\log {x\over m}|  \bigr)}
:={} \\
:= \Sigma_1+\Sigma_2+\Sigma_3.
\end{multline}

We have $\Sigma_1 \ll \sum_{m,n=1}^\infty {\taa(mn) / (mn)^{c}} \hm= F(c,c)$ and we will show below in~\eqref{eq:tau12-zeta-product} that
\begin{equation}\label{eq:sigma1-estimate}
F(c,c) \ll {1\over (c-1)^2} = \log^2 x.
\end{equation}

Further, for $x$ such that $|\log {x\over n}| \le 1$ we have $ |\log {x\over n}| \ge c |x-n| / x $ for~$c \hm= 1/(e-1)$. Then
$$
\Sigma_2 \ll \sum_{x/e \le n \le xe} \sum_m {\taa(mn) x \over (mn)^c |x-n|}.
$$
Note that $\taa(mn) \le \tau(mn) \le \tau(m) \tau(n)$, because $\tau$ is completely submultiplicative. Thus
$$
\Sigma_2 \ll x \sum_{x/e \le n \le xe} {\tau(n) \over n^c |x-n|} \sum_m {\tau(m) \over m^c}.
$$
Here
$$
\sum_{m=1}^\infty \tau(m) m^{-c} = \zeta^2(c) \ll (c-1)^{-2} = \log^2 x.
$$
Let $M(y) = \max_{n\le y} \tau(n)$. We have
$$
\Sigma_2 \ll x M(xe) \log^2 x \sum_{x/e \le n \le xe} {1\over n^c |x-n|},
$$
where the last sum is $\ll x^{-c} \log x \ll x^{-1} \log x$, so finally
\begin{equation}\label{eq:sigma2-estimate}
\Sigma_2 \ll M(xe) \log^3 x.
\end{equation}

Now consider $\Sigma_3$. Defining $M_{1,2}(y) = \max_{n\le y} \taa(n)$ we obtain
\begin{multline}\label{eq:sigma3-estimate}
\Sigma_3 \ll
\sum_{x/e \le n \le m \le xe}
{\taa(mn) x \over (mn)^c \min\bigl( |x-n|, |x-m| \bigr)}
\ll \\ \ll
{x M_{1,2}(x^2 e^2) \over x^{2c}}
\sum_{x/e \le n \le m \le xe}
\max\left( {|x-n|^{-1}}, {|x-m|^{-1}} \right)
\ll
M_{1,2}(x^2 e^2) \log x.
\end{multline}

Standard estimates \cite[Th.~315]{hardy2008} give
$M_{1,2}(y) \le M(y) \ll y^\eps$,
so substituting \eqref{eq:sigma1-estimate}, \eqref{eq:sigma2-estimate} and~\eqref{eq:sigma3-estimate} into \eqref{eq:sigma1234} we obtain
\begin{equation}\label{eq:f2-heart-estimate}
F_2^\heartsuit(c, x, T) \ll T^{-1} \bigl( M(xe) \log^3 x + M_{1,2}(x^2 e^2) \log x \bigr) \ll T^{-1} x^\eps.
\end{equation}

\medskip

Note also that by definition~\eqref{eq:sum-prime}
\begin{equation}\label{eq:sum-vs-suma}
\biggl|
\sum_{m,n\le x} \taa(mn) - \suma_{m,n\le x} \taa(mn)
\biggr|
\ll
\sum_{n\le x} \taa(\lfloor x \rfloor n) \ll M(x^2) x.
\end{equation}
Combining~\eqref{eq:perron-tau12}, \eqref{eq:f2-heart-estimate} and~\eqref{eq:sum-vs-suma} we get
\begin{multline}\label{eq:perron-tau12-rectified}
\sum_{m,n\le x} \taa(mn)
= {1\over(2\pi i)^2} \iint\limits_{[c-iT,c+iT]^2} F(z,w) {x^{z+w} \over zw} dz \, dw
+{} \\ +
O\bigl( x^{1+\eps} + T^{-1} x^{2+\eps} \bigr).
\end{multline}

\section{Double Dirichlet series for \texorpdfstring{$\taa$}{τ₁₂} }

Let us return to~\eqref{eq:tau12-series} and extract a product of zeta-functions from $F(z,w)$.
Define
\begin{equation}\label{eq:fxy-def}
f(x,y) = \sum_{a,b=0}^\infty \taa(p^{a+b}) x^a y^b,
\qquad
|x|, |y| < 1.
\end{equation}
Using identity
$$
\taa(p^a) - \taa(p^{a-1}) - \taa(p^{a-2}) + \taa(p^{a-3}) = 0
$$
multiply both sides of~\eqref{eq:fxy-def} by $(1-x)(1-x^2)$:
\begin{multline*}
(1-x)(1-x^2) f(x,y)
= \\ =
\sum_{a=3}^\infty \sum_{b=0}^\infty \bigl( \taa(p^{a+b}) - \taa(p^{a+b-1}) - \taa(p^{a+b-2}) + \taa(p^{a+b-3}) \bigr) x^a y^b
+{} \\ +
\sum_{b=0}^\infty y^b
\left(
(1-x-x^2)  \taa(p^b)
+ (1-x)  \taa(p^{b+1}) x
+ \taa(p^{b+2}) x^2
\right)
={} \\ =
\sum_{b=0}^\infty y^b
\left(
(1-x-x^2)  \taa(p^b)
+ (x-x^2)  \taa(p^{b+1}) x
+ x^2 \taa(p^{b+2})
\right)
\end{multline*}
and further
\begin{multline*}
(1-x)(1-x^2)(1-y)(1-y^2)f(x,y)
={} \\ =
(1-x-x^2) \bigl( (1-y-y^2) + (1-y)y + 2y^2 \bigr)
+{} \\ +
(x-x^2) \bigl( (1-y-y^2) + 2(1-y)y + 2y^2 \bigr)
+{} \\ +
x^2 \bigl( 2(1-y-y^2) + 2(1-y)y + 3y^2 \bigr)
={} \\ =
1+xy-x^2y-xy^2,
\end{multline*}
which induces
\begin{multline}\label{eq:fxy-closed-expression}
f(x,y) = {1+xy-x^2y-xy^2 \over (1-x)(1-x^2)(1-y)(1-y^2)}
={} \\ =
{1 - x^2 y - x y^2 - x^2 y^2 + x^3 y^2 + x^2 y^3 \over (1-x)(1-x^2)(1-y)(1-y^2)(1-xy)}.
\end{multline}
Representation~\eqref{eq:fxy-closed-expression} immediately implies that
\begin{multline}\label{eq:tau12-zeta-product}
F(z,w) = \prod_p f(p^{-z}, p^{-w})
= {\zeta(z) \zeta(2z) \zeta(w) \zeta(2w) \zeta(z+w) G(z,w)}
=\\=
{\zeta(z) \zeta(2z) \zeta(w) \zeta(2w) \zeta(z+w)
\over
\zeta(2z+w) \zeta(2w+z) } H(z,w),
\end{multline}
where series $H(z,w)$ converges absolutely in the region $\Re(2z+2w) \hm> 1 $. De\-fi\-ni\-te\-ly~$G(z,w)$ converges absolutely for $(z,w) \in Q := \{\Re z \hm\ge 1/3, \Re w \ge 1/3\}$.

Product of zeta-functions~\eqref{eq:tau12-zeta-product} shows that inside of the region $Q$ function~$F(z,w)$ has poles along lines $z=1$, $z=1/2$, $w=1$, $w=1/2$ and~$z+w=1$. All of them are of the first order, except poles at~$(1,1)$, $(1,1/2)$, $(1/2,1)$, which are of the second order, and a pole at~$(1/2, 1/2)$, which is of the third order.

\bigskip

Both \eqref{eq:tau-series} and \eqref{eq:tau12-zeta-product} are partial cases of a general rule, which will be stated as a lemma.

\begin{lemma}\label{l:tau1k-zeta-product}
Let $\tau_{1,k}(n) = \sum_{ab^k=n}1$. Then for $\Re z, \Re w > 1$ we have
\begin{equation}\label{eq:tau1k-zeta-product}
\sum_{m,n=1}^\infty {\tau_{1,k}(mn) \over m^z n^w}
= \zeta(z) \zeta(w)
{ \prod_{l=0}^k \zeta\bigl(lz+(k-l)w\bigr) \over
\prod_{l=1}^k \zeta\bigl(lz+(k+1-l)w\bigr) }
H_k(z,w),
\end{equation}
where the series $H_k$ converges absolutely for $\Re z, \Re w > 1/(k+2)$.
\end{lemma}
\begin{proof}
Cases $k=1$ and $k=2$ has been proven above, so we consider $k>2$ only. Let
$$
f(x,y) = \sum_{a,b=0}^\infty \tau_{1,k}(p^{a+b}) x^a y^b,
\qquad
|x|, |y| < 1.
$$
For a monomial $M$ let $[M]f(x,y)$ be a coefficient at $M$ in the series $f$.
Here
$$[x]f(x,y) \hm= [y]f(x,y) \hm= \tau_{1,k}(p) = 1,$$
so let us define
\begin{multline*}
g(x,y) = (1-x) (1-y) f(x,y)
= \\ =
\sum_{a,b=1}^\infty \bigl( \tau_{1,k}(p^{a+b}) - 2 \tau_{1,k}(p^{a+b-1}) + \tau_{1,k}(p^{a+b-2}) \bigr) x^a y^b
+ \\ +
\sum_{a=1}^\infty \bigl( \tau_{1,k}(p^{a}) - \tau_{1,k}(p^{a-1}) \bigr) (x^a + y^a) + 1.
\end{multline*}
We have
$$
\tau_{1,k}(p^a) = \begin{cases}
1, & a<k, \\
2, & k\le a<2k,
\end{cases}
$$
so one can verify that
$$
[x^a y^b] g(x,y) = \begin{cases}
 0, & a+b<k, \\
 1, & a+b=k, \\
 0, & a+b=k+1, ~ ab=0 \\
-1, & a+b=k+1, ~ ab>0.
\end{cases}
$$
Thus
$$
f(x,y)
= {1\over (1-x) (1-y)}
{ \prod_{l=1}^k (1-x^l y^{k+1-l}) \over
\prod_{l=0}^k (1-x^l y^{k-l}) }
h(x,y),
$$
where all monomials of the series $h(x,y)$ has degree at least $k+2$.
\end{proof}

\section{Path of integration and the main term}

Our aim is to translate the domain of integration in~\eqref{eq:perron-tau12-rectified} from $[c\hm-iT,c\hm+iT]^2$ till~$[b-iT,b+iT]^2$, where $b=1/3$. This is trickier than translating in the one-dimensional case, because a hyperrectangle $R$ with opposite vertices $(b\hm-iT,b-iT)$ and $(c+iT,c\hm+iT)$ has 24~two-dimensional faces. Figure~\ref{f:tesseract} contains a schematic plain projection of $R$ with 16~vertices and 32~edges marked.

\begin{figure}
\center\includegraphics{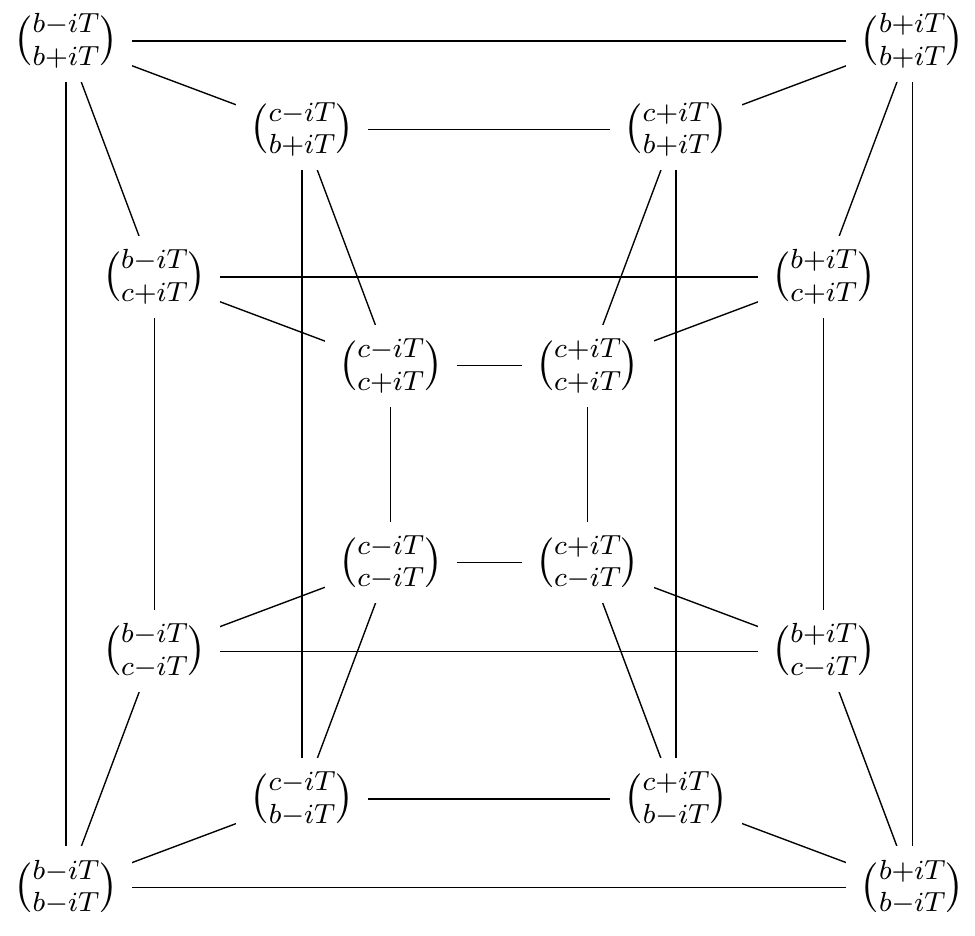}
\caption{The hyperrectangle $R$ with opposite vertices $(b-iT,b\hm-iT)$ and $(c+iT,c+iT)$}
\label{f:tesseract}
\end{figure}

Denote $L(z,w) = G(z,w) x^{z+w} z^{-1} w^{-1}$. This function has the same poles in~$R$ as~$G(z,w)$ has. Note that (on contrary with integration by one-dimensional contour) poles of the first order do not induce divergence of integrals by plane domains: e.~g., $\iint_{x^2+y^2\le1} {dx \, dy \over \sqrt{x^2+y^2}} = 2\pi < \infty$, however~$\int_{x^2\le 1} {dx\over x} = \infty$. Only poles of the second and higher orders are worth to pay attention.

Let $E(x)$ be the integral of $L(z,w)$ over all faces of $R$ except $[c\hm-iT, c+iT]^2$. By residue theorem \cite{shabat1992}
\begin{multline}\label{eq:cauchy-theorem}
{1\over(2\pi i)^2} \iint\limits_{[c-iT,c+iT]^2} L(z,w)\,dz\,dw
={} \\ =
\Bigl( \res_{z=w=1} + \res_{\scriptstyle z=1\atop \scriptstyle w=1/2} + \res_{\scriptstyle z=1/2 \atop \scriptstyle w=1} + \res_{z=w=1/2} \Bigr) L(z,w) + O(E(x)).
\end{multline}
Expanding $L(z,w)$ into Laurent series in two variables we get
\begin{gather}
\label{eq:res11}
\res_{z=w=1} L(z,w) = \zeta^3(2) G(1,1) x^2, \\
\label{eq:res12}
\res_{\scriptstyle z=1\atop \scriptstyle w=1/2} L(z,w) = \res_{\scriptstyle z=1/2\atop \scriptstyle w=1} L(z,w) = \zeta(2)\zeta(\half)\zeta(\threehalf) G(1,\half) x^{3/2}, \\
\label{eq:res22}
\res_{z=w=1/2} L(z,w) \ll x \log x.
\end{gather}
After substitution into~\eqref{eq:perron-tau12-rectified} the residue at~$(1/2,1/2)$ will be absorbed by error term, so it is enough to have only upper bound. Inserting~\eqref{eq:res11}, \eqref{eq:res12} and \eqref{eq:res22} into~\eqref{eq:cauchy-theorem} we get
\begin{equation}\label{eq:cauchy-theorem-2}
{1\over(2\pi i)^2} \iint\limits_{[c-iT,c+iT]^2} L(z,w)\,dz\,dw
=
C_1 x^2 + C_2 x^{3/2} + O(x\log x + E(x)),
\end{equation}
where
$$
C_1 = {\pi^6\over216} G(1,1),
\qquad
C_2 = {\pi^2\over3} \zeta(\half) \zeta(\threehalf) G(1,\half).
$$

Let us calculate numerical values of $C_1$ and $C_2$. Applying formal identity
$$
{F(z,w) \over \zeta(z) \zeta(w)}
=
\prod_p (1-p^{-z})(1-p^{-w}) \sum_{a,b=0}^\infty {\taa(p^{a+b}) \over p^{a+b}}
$$
at $z=w=1$ we get
\begin{equation*}\label{eq:C1-value}
C_1 = \res_{z=w=1} L(z,w)
= \prod_p (1-p^{-1})^2 \sum_{a,b=0}^\infty {\taa(p^{a+b}) \over p^{a+b}}
= 2.995\ldots
\end{equation*}
The product converges absolutely because
$$
(1-p^{-1})^2 \sum_{a,b=0}^\infty {\taa(p^{a+b}) \over p^{a+b}}
= \bigl( 1-2p^{-1}+O(p^{-2}) \bigr) \bigl( 1 + 2p^{-1}+O(p^{-2}) \bigr)
= 1 + O(p^{-2}).
$$
Similarly
$$
{F(z,w) \over \zeta(z) \zeta(w) \zeta(2z)}
=
\prod_p (1-p^{-z})(1-p^{-w})(1-p^{-2z}) \sum_{a,b=0}^\infty {\taa(p^{a+b}) \over p^{a+b}}
$$
implies
\begin{equation*}\label{eq:C2-value}
C_2 = 2
\!\!\!\!
\res_{\scriptstyle z=1\atop \scriptstyle w=1/2}
\!\!\!\!
L(z,w)
= 2 \zeta(1/2) \prod_p (1-p^{-1})^2 (1-p^{-1/2}) \sum_{a,b=0}^\infty {\taa(p^{a+b}) \over p^{a+b/2}}
= -5.404\ldots
\end{equation*}

\section{The error term}

Let us estimate $E(x)$. It was defined above to consist of integrals over~23 of~24~faces of the hyperrectangle~$R$, but due to the symmetry many of these integrals can be estimated in the same way.

In computations below we assume $ x^{1/2} \ll T \ll x$, the exact value of~$T$ will be specified later in~\eqref{eq:T-choice}.

There are 2~faces of form $[b-iT,b+iT]\times[c-iT,c+iT]$. We have
\begin{multline*}
I_1 := \int_{b-iT}^{b+iT} \int_{c-iT}^{c+iT} L(z,w)\,dz\,dw
\ll
\iint\limits_{[1,T]^2}
\zeta(b+it_1) \zeta (2b+2it_1)
\times{} \\ \times
\zeta(c+it_2) \zeta(2c+2it_2)
\zeta\bigl(b+c+i(t_1+t_2)\bigr) x^{b+c} t_1^{-1} t_2^{-1} dt_1 dt_2.
\end{multline*}
By~\eqref{eq:convexity} we can estimate
$$
\zeta(c+it_2) \zeta(2c+2it_2) \zeta\bigl(b+c+i(t_1+t_2)\bigr) \ll \log^{2/3} T \cdot 1 \cdot 1.
$$
As soon as $x^{1/\log x} \ll 1$ we have $x^{b+c} \ll x^{4/3}$. Also $\int_1^T t_2^{-1} dt_2 \ll \log T$. Thus $I_1$ can be estimated as
\begin{equation*}
I_1  \ll
x^{4/3} \log^{5/3} T
\int_1^T
\zeta(b+it) \zeta (2b+2it) t^{-1} dt.
\end{equation*}
By functional equation for $\zeta$, Lemma~\ref{l:zeta-square} and Lemma~\ref{l:log-sum}
\begin{multline}\label{eq:holder}
J := \int_1^T
\zeta(b+it) \zeta (2b+2it) t^{-1} dt
\ll
\int_1^T t^{1/6} \zeta^2(2/3+it) t^{-1} dt
\ll T^{1/6} \log T.
\end{multline}
Then
\begin{equation}\label{eq:int-b-c-b+c+}
I_1 \ll x^{4/3} T^{1/6} \log^{8/3} T.
\end{equation}

We will show below in~\eqref{eq:error-term} that integrals over other faces (and so~$E(x)$ as a whole) are less than either $I_1$ or $x^{2+\eps} T^{-1}$, so $T$ should be chosen to equalize this two magnitudes:
\begin{equation}\label{eq:T-choice}
T = x^{4/7}.
\end{equation}
Substitute it into~\eqref{eq:perron-tau12-rectified} and \eqref{eq:cauchy-theorem-2} to obtain the final error term $x^{10/7+\eps}$, which approves the statement of the Theorem~\ref{th:main-theorem}.

\medskip

From here and till the end of the section we will omit factors $\ll x^\eps$ in asymptotic estimates for the brevity: they do not influence the resulting error term.

There are 4 faces of form $[b-iT,b+iT]\times[b\pm iT, c\pm iT]$. We have
\begin{multline*}
I_2 := \int_{b-iT}^{b+iT} \int_{b+iT}^{c+iT} L(z,w)\,dz\,dw
\ll
\int_1^T \int_b^c \zeta(b+it) \zeta(2b+2it)
\times{} \\ \times
\zeta(\sigma+iT) \zeta(2\sigma+2iT) \zeta\bigl( b+\sigma+i(t+T) \bigr)
x^{b+\sigma} t^{-1} T^{-1} d\sigma\,dt
\ll{} \\ \ll
x^{1/3} J T^{-1}
\max_{\scriptstyle  \sigma\in[b,c] \atop \scriptstyle  t\in[1,T]}
\zeta(\sigma+iT) \zeta(2\sigma+2iT) \zeta\bigl(b+\sigma+i(t+T)\bigr)  x^\sigma
\ll{} \\ \ll
x^{1/3} T^{-5/6}
\max_{\sigma\in[b,c]}
\zeta(\sigma+iT) \zeta(\sigma+1/3+iT) \zeta(2\sigma+iT) x^\sigma.
\end{multline*}
Splitting $[b,c]$ into intervals $[1/3,1/2]$, $[1/2,2/3]$, $[2/3,c]$ and estimating~$\zeta(\sigma\hm+iT) \* \zeta(\sigma+1/3+iT) \* \zeta(2\sigma+iT) x^\sigma$ on each of them separately, we get
$$
I_2
\ll
x^{1/3} T^{-5/6}
( T^{\mu(1/3)+2\mu(2/3)} x^{1/2} + T^{\mu(1/2)+\mu(5/6)} x^{2/3} + T^{\mu(2/3)} x).
$$
Utilizing rough estimate $\mu(1/2)\le 1/6$ from \cite[Th. 5.5]{titchmarsh1986} we get by~\eqref{eq:convexity} that
\begin{equation}\label{eq:mu-rough}
\mu(\sigma) \le \begin{cases}
	1/2 - 2\sigma/3, & \sigma \in [0,1/2], \\
	(1-\sigma)/3, & \sigma \in [1/2,1]
	\end{cases}
\end{equation}
and
\begin{equation}\label{eq:mu1/3-mu2/3}
\mu(1/3)\le 5/18, \quad \mu(2/3)\le 1/9, \quad \mu(5/6)\le 1/18,
\end{equation}
 so
\begin{equation}\label{eq:int-b-b+b+c+}
I_2 \ll x^{1/3} T^{-5/6} (T^{1/2} x^{1/2} + T^{2/9}x^{2/3} + T^{1/9} x) \ll x^{4/3}.
\end{equation}

\medskip

There is 1 face of form $[b-iT,b+iT]^2$. Applying~\eqref{eq:mu1/3-mu2/3} we have
\begin{multline*}
I_3 := \iint\limits_{[b-iT,b+iT]^2} L(z,w)\,dz\,dw
\ll
\iint\limits_{[1,T]^2} \zeta(b+it_1) \zeta (2b+2it_1)
\times{} \\ \times
\zeta(b+it_2) \zeta(2b+2it_2) \zeta\bigl(2b+i(t_1+t_2)\bigr) x^{2b} t_1^{-1} t_2^{-1} dt_1 dt_2
\ll{} \\ \ll
x^{2/3} \iint\limits_{[1,T]^2}
t_1^{5/18+1/9-1} t_2^{5/18+1/9-1} (t_1+t_2)^{1/9} dt_1 dt_2,
\end{multline*}
which implies
\begin{equation}\label{eq:int-b-b-b+b+}
I_3 \ll
x^{2/3} T^{8/9},
\end{equation}
which is less than $x^{4/3}$ by our choice of $T$ in \eqref{eq:T-choice}.

\medskip

There are 4 faces of form $[c-iT,c+iT]\times[b\pm iT, c\pm iT]$. We have
\begin{multline}\label{eq:int-c-b+c+c+-preliminary}
I_4 := \int_{c-iT}^{c+iT} \int_{b+iT}^{c+iT} L(z,w)\,dz\,dw
\ll{} \\ \ll
\int_1^T \int_b^c \zeta(c+it) \zeta(2c+2it) \zeta(\sigma+iT) \zeta(2\sigma+2iT) \zeta\bigl( c+\sigma+i(t+T) \bigr)
\times{} \\ \times
x^{c+\sigma} t^{-1} T^{-1} d\sigma\,dt
\ll
x T^{-1} \int_b^c \zeta(\sigma+iT) \zeta(2\sigma+2iT) x^\sigma d\sigma.
\end{multline}
Here
$$
\int_b^c \zeta(\sigma+iT) \zeta(2\sigma+2iT) x^\sigma d\sigma
\ll
\max_{\sigma\in[b,c]} \zeta(\sigma+iT) \zeta(2\sigma+iT) x^\sigma.
$$
For $\sigma\in[b,1/2]$ we have
\begin{equation}\label{eq:int-c-b+c+c+-helper1}
\zeta(\sigma+iT) \zeta(2\sigma+iT) x^\sigma \ll T^{\mu(1/3)+\mu(2/3)} x^{1/2} \ll T x^{1/3}.
\end{equation}
Taking into account \eqref{eq:mu-rough} for $\sigma\in[1/2,1]$ we get
\begin{equation}\label{eq:int-c-b+c+c+-helper2}
\zeta(\sigma+iT) \zeta(2\sigma+iT) x^\sigma \ll T^{\mu(\sigma)} x^\sigma \ll x^{\mu(\sigma)+\sigma} \ll x^{(1+2\sigma)/3} \ll x.
\end{equation}
Returning to \eqref{eq:int-c-b+c+c+-preliminary} we get
\begin{equation}\label{eq:int-c-b+c+c+}
I_4
\ll x^2 T^{-1} + x^{4/3}.
\end{equation}

\medskip

There are 4 faces of form $[b\pm iT, c\pm iT]^2$. We have
\begin{multline}\label{eq:int-b+b+c+c+}
I_5 := \iint\limits_{[b+iT, c+iT]^2} L(z,w)\,dz\,dw
\ll
\max_{(z,w) \in [b+ iT, c+ iT]^2} L(z,w)
\ll{} \\ \ll
\max_{\sigma_1,\sigma_2 \in [b,c] } \zeta(\sigma_1+iT) \zeta(2\sigma_1+2iT) \zeta(\sigma_2+iT) \zeta(2\sigma_2+iT) \zeta(\sigma_1+\sigma_2+2iT)
\times{} \\ \times
x^{\sigma_1+\sigma_2} T^{-2}
\ll
T^{2\mu(1/3)+3\mu(2/3)-2} x^2 \ll x^2 T^{-1}.
\end{multline}

\medskip

Finally, there are 8 faces, which are parallel either to $z$- or $w$-plane, of form $[b-iT,c+iT] \times w$, where $w\in W:=\{b\pm iT, c\pm iT\}$. We have
\begin{multline*}\label{eq:i6-preliminary}
I_6 :=
\iint\limits_{b-iT}^{~~~~c+iT}
L(z,b+iT) \, dz
\ll
\int_1^T
\int_b^c \zeta(\sigma+it) \zeta(2\sigma+2it) \zeta\bigl(\sigma+b+i(t+T)\bigr)
\times{} \\ \times
\zeta(b+iT) \zeta(2b+2iT)
x^{\sigma+b} t^{-1} T^{-1} d\sigma \, dt
\ll
T^{\mu(1/3)+\mu(2/3)-1} x^{1/3}
\times{} \\ \times
\int_1^T \int_b^c \zeta(\sigma+it) \zeta(2\sigma+2it) \zeta(\sigma+1/3+iT) x^\sigma t^{-1} d\sigma \, dt.
\end{multline*}
Here
$$
\zeta(\sigma+it) \zeta(2\sigma+2it) \zeta(\sigma+1/3+iT) x^\sigma t^{-1}
\ll
T^{\mu(1/3)+2\mu(2/3)-1} x,
$$
so
\begin{equation}\label{eq:int-b-c+1}
I_6 \ll T^{\mu(1/3)+\mu(2/3)-1} x^{1/3} \int_1^T T^{\mu(1/3)+2\mu(2/3)-1} x \, dt
\ll x^{4/3}.
\end{equation}

Also
\begin{multline*}
I_7 := \iint\limits_{b-iT}^{~~~~c+iT} L(z,c+iT) \, dz
\ll
\int_1^T \int_b^c
\zeta(\sigma+it) \zeta(2\sigma+2it)
\times{} \\ \times
\zeta\bigl(\sigma+c+i(t+T)\bigr)
\zeta(c+iT) \zeta(2c+2iT)
x^{\sigma+c} t^{-1} T^{-1} d\sigma \, dt
\ll{} \\ \ll
x T^{-1} \int_{1}^T \int_b^c \zeta(\sigma+it) \zeta(2\sigma+2it) x^{\sigma} t^{-1} d\sigma\, dt
\end{multline*}
We derive from \eqref{eq:int-c-b+c+c+-helper1} and \eqref{eq:int-c-b+c+c+-helper2} that
$$
\int_b^c \zeta(\sigma+it) \zeta(2\sigma+2it) x^{\sigma} d\sigma
\ll t x^{1/3} + x,
$$
so
\begin{equation}\label{eq:int-b-c+2}
I_7 \ll
x T^{-1} \int_{1}^T (x^{1/3} + x t^{-1}) dt
\ll
x^2 T^{-1} + x^{4/3}.
\end{equation}

\medskip

Now summing up \eqref{eq:int-b-c-b+c+}, \eqref{eq:int-b-b+b+c+}, \eqref{eq:int-b-b-b+b+}, \eqref{eq:int-c-b+c+c+}, \eqref{eq:int-b+b+c+c+}, \eqref{eq:int-b-c+1}, \eqref{eq:int-b-c+2} we get
\begin{equation}\label{eq:error-term}
E(x) \ll x^{4/3} T^{1/6} + x^{2+\eps} T^{-1}.
\end{equation}

\section{Conclusion}

Our result can be slightly improved under the Riemann hypothesis. In such case we have $\zeta^{\pm1}(s) \ll x^\eps$ for $\sigma > 1/2$ and $\mu(1/2)=0$ due to \cite[(14.2.5)--(14.2.6)]{titchmarsh1986}. Then~\eqref{eq:tau12-zeta-product} immediately induces $F(z,w) \hm\ll x^\eps \* \zeta(z) \* \zeta(w)$ for $\Re z, \Re w \hm> 1/4$ and all double integrals, incorporated in~$E(x)$, can be split and estimated by a product of two one-dimensional integrals. For $b=1/4+1/\log x$ we obtain
\begin{align*}
\int_{b-iT}^{b+iT} \zeta(z) {x^z\over z} dz &\ll x^{1/4+\eps} T^{1/4}, \\
\int_{c-iT}^{c+iT} \zeta(z) {x^z\over z} dz &\ll x^{1+\eps}, \\
\int_{b\pm iT}^{c\pm iT} \zeta(z) {x^z\over z} dz &\ll (x^{1/2+\eps} T^{1/4} + x^{1+\eps})/T.
\end{align*}
Then $E(x) \ll x^{5/4+\eps} T^{1/4}$ and choice $T=x^{3/5}$ provides us with $\alpha\hm=7/5\hm=1.4$ in the statement of Theorem~\ref{th:main-theorem}.

\medskip

One should expect in the view of \eqref{eq:tau1k-zeta-product} that
\begin{equation}\label{eq:tau1k-sum}
\sum_{m,n\le x} \tau_{1,k}(mn) = D_1 x^2 + D_2 x^{1+1/k} + O(x^{\alpha_k+\eps}).
\end{equation}
Translating the domain of integration till $[b-iT,b+iT]^2$, where $b\hm=1/(k+1)$, leads to the error term at least $x^{{k+2\over k+1}+\eps} T^{{1\over2}-{1\over k+1}} + x^{2+\eps} T^{-1}$, which corresponds to~$\alpha_k \hm= (4k+2)/(3k+1)$ for the best possible choice of~$T$. Under the Riemann hypothesis for $b=1/2k+1/\log x$ we obtain $\alpha_k \hm= (4k-1)/(3k-1)$. However, for~$k\hm>2$ both of these estimates are bigger than~$x^{4/3}$  and  absorbs the term $ D_2 x^{1+1/k}$ in~\eqref{eq:tau1k-sum}. Such result can hardly be reckoned satisfactory.

\medskip

One can consider the exponential divisor function $\taue$, which is multiplicative and defined by $\taue(p^a) = \tau(a)$. As far as $\taue(p^k)=\taa(p^k)$ for $k=1,2,3,4$, the Dirichlet series for $\taue$ also possesses the representation~\eqref{eq:tau12-zeta-product}, so Theorem~\ref{th:main-theorem} remains valid for~$\taue$ instead of $\taa$.

\bibliographystyle{ugost2008s}
\bibliography{taue}

\begin{thebibliography}{1}
\def\selectlanguageifdefined#1{
\expandafter\ifx\csname date#1\endcsname\relax
\else\language\csname l@#1\endcsname\fi}
\providecommand*{\href}[2]{{\small #2}}
\providecommand*{\url}[1]{{\small #1}}
\providecommand*{\BibUrl}[1]{\url{#1}}
\providecommand{\BibAnnote}[1]{}
\providecommand*{\BibEmph}[1]{#1}
\providecommand*{\cyrdash}{\hbox to.8em{--\hss--}}
\providecommand*{\BibDash}{\ifdim\lastskip>0pt\unskip\nobreak\hskip.2em\fi
\cyrdash\hskip.2em\ignorespaces}

\bibitem{balazard2008}
\selectlanguageifdefined{english}
\BibEmph{Balazard~M., Naimi~M., Pétermann~Y.-F.~S.} Étude d'une somme
  arithmétique multiple liée à la fonction de Möbius~// \BibEmph{Acta
  Arith.} \BibDash
\newblock 2008. \BibDash
\newblock Vol. 132, no.~2. \BibDash
\newblock P.~245--298.

\bibitem{graham1988}
\selectlanguageifdefined{english}
\BibEmph{Graham~S.~W., Kolesnik~G.} On the difference between consecutive
  squarefree integers~// \BibEmph{Acta Arith.} \BibDash
\newblock 1988. \BibDash
\newblock Vol.~49, no.~5. \BibDash
\newblock P.~435--447.

\bibitem{hardy2008}
\selectlanguageifdefined{english}
\BibEmph{Hardy~G.~H., Wright~E.~M.} An introduction to the theory of numbers~/
  Ed.\ by\ D.~R.~Heath-Brown, J.~H.~Silverman. \BibDash
\newblock $6^{\rm th}$, rev. edition. \BibDash
\newblock New York~: Oxford University Press, 2008. \BibDash
\newblock xxi+635~p. \BibDash
\newblock ISBN:~\href{http://isbndb.com/search-all.html?kw=0199219869,
  9780199219865}{0199219869, 9780199219865}.

\bibitem{huxley2005}
\selectlanguageifdefined{english}
\BibEmph{Huxley~M.~N.} Exponential sums and the Riemann zeta function~V~//
  \href{http://dx.doi.org/10.1112/S0024611504014959}{\BibEmph{Proc. Lond. Math.
  Soc.}} \BibDash
\newblock 2005. \BibDash
\newblock Vol.~90, no.~1. \BibDash
\newblock P.~1--41.

\bibitem{ivic2003}
\selectlanguageifdefined{english}
\BibEmph{Ivić~A.} The Riemann zeta-function: Theory and applications. \BibDash
\newblock Mineola, New York~: Dover Publications, 2003. \BibDash
\newblock 562~p. \BibDash
\newblock ISBN:~\href{http://isbndb.com/search-all.html?kw=0486428133,
  9780486428130}{0486428133, 9780486428130}.

\bibitem{shabat1992}
\selectlanguageifdefined{english}
\BibEmph{Shabat~B.~V.} Introduction to complex analysis II: Functions of
  several variables~/ Ed.\ by\ S.~Ivanov. Vol. 110 of Translations of
  mathematical monographs. \BibDash
\newblock Providence, Rhode Island~: American Mathematical Soc., 1992. \BibDash
\newblock x+371~p. \BibDash
\newblock ISBN:~\href{http://isbndb.com/search-all.html?kw=082189739X,
  9780821897393}{082189739X, 9780821897393}.

\bibitem{titchmarsh1986}
\selectlanguageifdefined{english}
\BibEmph{Titchmarsh~E.~C.} The theory of the Riemann zeta-function~/ Ed.\ by\
  D.~R.~Heath-Brown. \BibDash
\newblock $2^{\rm nd}$, rev. edition. \BibDash
\newblock New-York~: Oxford University Press, 1986. \BibDash
\newblock 418~p. \BibDash
\newblock ISBN:~\href{http://isbndb.com/search-all.html?kw=0198533691,
  9780198533696}{0198533691, 9780198533696}.

\end{thebibliography}

\end{document}